\documentclass{amsart}
\usepackage{amsmath}
\usepackage{graphicx}
\usepackage{amssymb}
\newcommand{\R}{\mathbb R}
\newcommand{\E}{\mathbb E}
\newcommand{\Z}{\mathbb Z}
\newcommand{\N}{\mathbb N}

\newtheorem{theorem}{Theorem}[section] 
\newtheorem{lemma}[theorem]{Lemma}     

\newtheorem{definition}[theorem]{Definition}

\title{$n$-digit Benford converges to Benford}

\author{Azar Khosravani}
\address{Department of Science and Mathematics, Columbia College Chicago,
	Chicago, IL 60605}
\email{akhosravani@colum.edu}

\author{Constantin Rasinariu}
\address{Department of Science and Mathematics, Columbia College Chicago,
	Chicago, IL 60605}
\email{crasinariu@colum.edu}


\begin{document}
	\maketitle
	
\begin{abstract}
		Using the sum invariance property of Benford random variables, we prove 
		that an $n$-digit Benford variable converges to a Benford variable as 
		$n$ approaches infinity.
\end{abstract}

\section{Introduction}
Given a positive real number $y$, and a positive integer $i$, we define 
$D_i(y)$ as the $i$-th significant digit of $y$, where $D_1:\mathbb{R}^{\,+}\to 
\{1,\ldots,9 \}$ and $D_i:\mathbb{R}^{\,+} \to \{0,1,\ldots,9 \}$ for $i>1$. 
Thus, $D_1(2.718)=2$ and $D_3(2.718)=1$. We assume base 10 throughout this paper.

Let $\mathcal A$ be the smallest sigma algebra generated by $D_i$.
Then $D_i^{-1}(d) \in \mathcal A$ for all $i$ and $d$.
Within this framework, a random variable $Y$ is \textit{Benford} 
\cite{Benford-1938,Hill-1995,Berger-2011} if for all $m \in 
\mathbb{N}$, $d_1\in \{1,\ldots,9 \}$ and $d_i\in \{0,1,\ldots,9 \}$ 
for $i>1$, the probability that the first $m$ digits of a real number are $d_1d_2\cdots d_m$ is given by
\begin{equation}
	\label{e:ben}
	P\Big(D_1(Y) = d_1,\ldots,D_m(Y) = d_m\Big)=\log \bigg(1+ 
	\Big(\sum_{j=1}^{m} 10^{m-j}d_j\Big)^{-1}\bigg)~.
\end{equation}

While Benford variables have logarithmic distributions in all of their digits, 
often times, in Benford literature the focus has only been on the distribution of the first digit. Such a limitation may obscure the true nature of the quantity investigated. There are data sets which exhibit a perfect ``Benford'' distribution in the first digit, but fail to do so in the second. Nigrini \cite{Nigrini-2012} provided such an example, and consequently recommended the use of the first two digit test in order to improve the recognition of the Benford datasets, and thus to identify financial fraud. He also recommended this approach for other accounting related  analysis.

Such cases were generalized in \cite{Khosravani-2013}, where  a new class of random variables, called \textit{$n$-digit Benford variables}, was introduced. These variables exhibit a logarithmic digit distribution  only in their first $n$ digits, but are not guaranteed to be logarithmically distributed beyond the $n$-th digit. Unlike Benford variables whose decimal logarithm is uniformly distributed \mbox{mod 1}, the decimal logarithm of  n-digit Benford random variables has less stringent constraints; it must only satisfy 
prescribed areas over a given partition of the unit interval. This provides us 
with a collection of random variables that contains the Benford variables as a subset. 

It is intuitive to assume that when $n$ goes to infinity, a $n$-digit Benford variable converges to Benford. The purpose of this paper is to prove that this is indeed the case.  

This paper is structured as follows: in the next section we introduce the 
$n$-digit Benford variables 
together with some of their properties. In section \ref{sed:SI} we briefly 
discuss sum invariance, which is fundamental for our 
main result.  Finally, using sum invariance, in section \ref{sec-main} we show that an 
$n$-digit Benford 
variable converges to Benford, as $n\to\infty$.

\section{$n$-digit Benford}
\label{sec:nBen}

An $n$-digit Benford random variable behaves as a Benford variable only in the 
first $n$-digits, but may not have a logarithmic digit distribution beyond the 
$n$th digit \cite{Khosravani-2013}.
\begin{definition}
	Let $n\in\N$. A random variable $Y$ is $n$-digit Benford if for all  
	$d_1\in 
	\{1,\ldots,9 \}$ and all $d_i\in \{0,1,\ldots,9 \}$, for $2\le i \le n$
	\begin{equation}
		\label{eq:nBen}
		P\Big(D_1(Y) = d_1,\ldots,D_n(Y) = d_n\Big)=\log \bigg(1+ 
		\Big(\sum_{j=1}^{n} 10^{n-j}d_j\Big)^{-1}\bigg)~.
	\end{equation}
\end{definition}
\noindent
Note that a Benford variable is an $n$-digit Benford variable, for any $n$.
\begin{lemma}
	\label{kBen}
	If $Y$ is $n$-digit Benford, then it is a $k$-digit 
	Benford, 
	for all $1 \le k < n$.
\end{lemma}
\begin{proof}
Let $k=n-1$. Then, by (\ref{eq:nBen}) 
\begin{eqnarray*}
 &&P\Big(D_1(Y)= d_1,\ldots,D_{n-1}(Y) = d_{n-1}\Big)\\
	&&= \sum\limits_{d_n=0}^{9}P\Big(D_1(Y) = d_1,\ldots,D_{n-1}(Y) = d_{n-1},D_n(Y) = d_n\Big) \\
	&&=\sum_{d_n=0}^{9}
	\log \bigg(1+ 
	\Big(\sum_{j=1}^{n} 10^{n-j}d_j\Big)^{-1}\bigg)\\
	&&=\log \Big( 
	\frac{10^{n-1}d_1+\cdots+10\, d_{n-1}+1}{10^{n-1}d_1+\cdots+10\, d_{n-1}}	\times \cdots \times
	\frac{10^{n-1}d_1+\cdots+10\,d_{n-1}+10}{10^{n-1}d_1+\cdots+10\, d_{n-1}+9}
	\Big)\\
	&&=\log\Big( 
	\frac{10^{n-1}d_1+\cdots+10\,d_{n-1}+10}{10^{n-1}d_1+\cdots+10\, d_{n-1}}
	\Big) = 
	\log \bigg(1+ 
	\Big(\sum_{j=1}^{n-1} 10^{n-j}d_j\Big)^{-1}\bigg)
	~.
\end{eqnarray*}
\end{proof}

As an example, let us consider the $2$-digit Benford variable $Y$ with the 
probability 
density function given by
\begin{equation}
\label{e:2-ben}
f(y) = \begin{cases}
\frac{\pi}{2 y \ln 10} 
\sin \left(\pi \beta_{d_1 d_2}(y)\right), 
&  d_1+\frac{d_2}{10} \le y < d_1+\frac{d_2+1}{10} \\
0, ~&\text{otherwise}
\end{cases}
\end{equation}
where $\beta_{d_1 d_2}(y)=( \log \frac{10 y}{10 d_1 + d_2})/(\log \frac{10 d_1 + d_2 
+1}{10 d_1 + d_2})$. Its graph is illustrated in figure (\ref{fig:2digit-Benford}).
\begin{figure}[tbh]
\centering
\includegraphics[width=\linewidth]{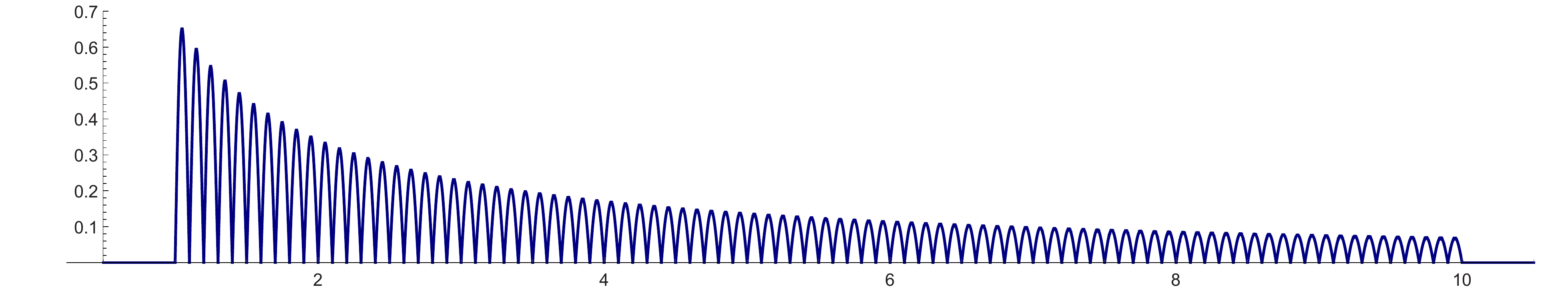}
\caption{The pdf of a $2$-digit Benford variable}
\label{fig:2digit-Benford}
\end{figure}
We can check that:
$P(D_1(Y) = d_1, D_2(Y) = d_2)=\log \Big(1+ (10\, d_1 + d_2)^{-1}\Big)$. From 
lemma \ref{kBen} this is a 1-digit Benford variable as well. However, $Y$ is 
not a $3$-digit 
Benford variable, since for example 
$P(D_1(Y)=1,D_2(Y)=1,D_3(Y)=1)= 2.86\times 10^{-3}$ instead of $3.89\times 
10^{-3}$ as required by 
(\ref{eq:nBen}).

\section{Sum invariance}
\label{sed:SI}

To define sum invariance, we first define the significand 
function, also known as the mantissa function.

\begin{definition}
	The significand function  $S:\R^+ \to [1,10)$ is defined as
	\begin{equation*}
	S(x)=10^{\log x - \lfloor{\log x }\rfloor}~,
	\end{equation*}
where $\lfloor x \rfloor$ denotes the floor of $x$.
\end{definition}
Let us consider a finite collection of positive real numbers ${K}$, and define 
$S_{d_1\cdots d_n}$ 
to be the sum of the significands of the numbers starting with the sequence of 
digits $d_1\cdots d_n$. Sum invariance means that $S_{d_1\cdots d_n}$ is 
digit independent. For instance, consider the Fibonacci sequence which is known 
to be Benford \cite{Duncan-1967}. Then for the first $50000$ Fibonacci numbers 
we obtain

\begin{table}[htb]
	\label{tab:fib}
{\footnotesize 
\begin{tabular}{|c|c|c|c|c|c|c|c|c|c|}
	\hline  
	$S_1$ & $S_2$  & $S_3$ & $S_4$ & $S_5$ & $S_6$  & $S_7$  & $S_8$ & $S_9$  	\\ 
	\hline  
	21714.0 & 21712.2  & 21717.8  & 21707.4 & 21713.2 & 21725.0 & 21702.7 & 21717.4 & 21715.5   \\ 
	\hline 
\end{tabular} 
}
\vspace{1ex}
\caption{Sum invariance illustration for the first $50000$ Fibonacci 
numbers}
\end{table}
\vspace{-4ex}
\noindent
where $S_1$ denotes the sum of all significands starting with $1$, etc.

Nigrini was the first to notice sum invariance in some large collections of 
data \cite{Nigrini-1992}. Allaart \cite{Allaart-1997} refined this concept, by 
defining it in connection with continuous random variables.  Specifically, a 
distribution is sum invariant if the expected value of the significands of all 
entries starting with a fixed $n$-tuple of leading significant digits is the 
same as for any other $n$-tuple: $\E\left[S_{d_1\cdots d_n} Y\right]=	
\E\left[ S_{d_1'\cdots d_n'} Y\right]$. Allaart showed that a random variable 
is sum invariant if and only if it is Benford. Berger \cite{Berger-2011} proved 
that for sum invariant random variables
\begin{equation}
\label{ExpectedSum}
	\E\left[ S_{d_1\cdots d_n} Y\right] = \frac{10^{1-n}}{\ln 10}.
\end{equation}
For example, for a Benford sequence with $50000$ elements, formula  
(\ref{ExpectedSum}) yields $S_1=\cdots=S_9=21714.7$ rounded to 
the tenths, which is very close to the actual values for the Fibonacci numbers 
illustrated in table 1. Naturally, the more numbers are taken from the 
sequence, the closer one gets to the theoretical sum.

\section{Main result}
\label{sec-main}

A random variable is sum invariant if and only if it is Benford 
\cite{Allaart-1997,Berger-2011}.  Using this result, we will prove that an 
$n$-digit 
Benford variable converges to Benford as $n$ approaches infinity 
by calculating the bounds for the expected value of its significand.

Given a function $g:\R\to\R$, we define  $g^{\dagger}:\R \to[0,1)$ as
\[
g^{\dagger}(x) = \begin{cases}
\sum_{k\in \Z}g(x+k), & \forall x\in [0,1) ,\\
0, ~&\textrm{otherwise.}
\end{cases}
\]

\begin{lemma}
\label{expected}
Let $Y$ and $X=\log Y$ be two random variables with the probability density 
functions $f$ and $g$, respectively. Then
\begin{equation}
\label{eq:sum}
	\E\left[ S_{d_1\cdots d_m} Y\right] = 
	\int_{\log(d_1+\cdots \frac{d_m}{10^{m-1}})}^{\log(d_1+\cdots \frac{d_m+1}{10^{m-1}})}
	\!\!\!
	10^x\,g^{\dagger}(x)\,dx ~.
\end{equation}

\end{lemma}
\begin{proof}
Using $f(y)=g(\log y)/(y\ln 10)$, we get
	\begin{eqnarray*}
	\E\left[ S_{d_1\cdots d_m} Y\right] &=& \int_{-\infty}^{\infty} 
	S_{d_1\cdots 
	d_m} (y) f(y) dy \\
	& = & \sum_{k\in \Z} 
	\int_{10^k(d_1+\cdots+\frac{d_m}{10^{m-1}})}^{10^k(d_1+\cdots+\frac{d_m+1}{10^{m-1}})}
	y\,10^{-k}\, \frac{g (\log y)}{y \ln 10} dy \\
	& = & 
	\int_{\log(d_1+\cdots+\frac{d_m}{10^{m-1}})}^{\log(d_1+\cdots+\frac{d_m+1}{10^{m-1}})}
	10^{x} \sum_{k\in \Z} g(x+k)\,dx~.
	\end{eqnarray*}
\end{proof}

It is known that a necessary and sufficient 
condition for a random variable to be Benford is that $g^{\dagger}=1$  
\cite{Diaconis-1977,Berger-2011}. 
Consequently, equation (\ref{ExpectedSum}) follows immediately from lemma 
\ref{expected}.

There are arbitrary many ways in which we can build a $n$-digit Benford 
variable. Let $\mathcal{B}_n$ be the infinite collection of all $n$-digit 
Benford variables. We use $\E\left[ S_{d_1\cdots d_n} \mathcal{B}_n\right]$ to 
denote the collection of the expected values of the significands of the 
elements of  $\mathcal{B}_n$. 
The next theorem leads to the main result of our 
paper. It provides the bounds for the expected value $\E\left[S_{d_1\cdots d_n} 
Y\right]$ for $Y \in \mathcal{B}_n$.  

\begin{theorem}
\label{n-dig-Ben}
	Let $Y \in \mathcal{B}_n$. 
	Then
\begin{equation}
\label{eq:theo}
10^{1-n} 
\log\left( 1 + \frac{1}{x_n} \right)^{x_n}  \le 
\E\left[ S_{d_1\cdots d_n} Y\right] \le
10^{1-n} 
\log\left( 1 + \frac{1}{x_n} \right)^{x_n+1}
\end{equation}
where $x_n = 10^{n-1}d_1 + \cdots + d_n$.
\end{theorem}
\begin{proof}
We will calculate the lower and upper bounds of  $\E\left[ S_{d_1\cdots d_n} 
Y\right]$ using the fact that 
$\int_0^s g^{\dagger}(x) dx$ is monotonically increasing  with $s$, where $g$ 
is the probability density function of $\log Y$. 
From lemma \ref{expected},  we obtain
\begin{eqnarray*}
\E\left[ S_{d_1\cdots d_n} Y\right] &=& (d_1+\cdots + 
\frac{d_n+1}{10^{n-1}})
\log\left(d_1+\cdots \frac{d_n+1}{10^{n-1}}\right) \\
&&- (d_1+\cdots \frac{d_n}{10^{n-1}})\log\left(d_1+\cdots + 
\frac{d_n}{10^{n-1}}\right) \\
&& - \int_{\log(d_1+\cdots \frac{d_n}{10^{n-1}})}^{\log(d_1+\cdots 
\frac{d_n+1}{10^{n-1}})}
 \!\!\! 10^s \ln 10  \int_0^s
 g^{\dagger}(x)\,dx \,ds.
\end{eqnarray*}
Since $Y \in \mathcal{B}_n$, we get
\begin{equation}
\label{eq:ints}
\int_0^s g^{\dagger}(x) dx = \log\left(d_1+\cdots + \frac{d_n}{10^{n-1}}\right) 
+ \int_{\log(d_1+\cdots + 
\frac{d_n}{10^{n-1}})}^s g^{\dagger}(x) dx ~.
\end{equation}
The second term in (\ref{eq:ints}) can take any value between $0$ and 
$\log(1+1/(10^{n-1}d_1+\cdots+d_n))$, since $g^{\dagger}(x)$ is only 
constrained by its total area over the interval 
$$
\Big[\log(d_1+\cdots + \frac{d_n}{10^{n-1}}),~\log(d_1+ \cdots + \frac{d_n+1}{10^{n-1}})\Big].
$$ 
It follows that 
\begin{equation}
\label{eq:min}
10^{1-n} 
\log\left( 1 + \frac{1}{x_n} \right)^{x_n} \le \E\left[S_{d_1\cdots d_n} 
Y\right]~, \forall\, Y \in \mathcal{B}_n
\end{equation}
where $x_n = 10^{n-1}d_1 + \cdots + d_n$. 
Similarly we obtain
\begin{equation}
\label{eq:max}
\E\left[ S_{d_1\cdots d_n} Y\right] \le 10^{1-n} 
\log \left( 1 + \frac{1}{x_n} \right)^{x_n} 
+ 10^{1-n} \log\left( 1 + \frac{1}{x_n} \right)~, \forall\, Y \in \mathcal{B}_n
\end{equation}
which completes the proof.
\end{proof}

As $n\to \infty$, both lower and upper bounds of $\E\left[ S_{d_1\cdots d_n} 
\mathcal{B}_n\right]$ approach $\frac{10^{1-n}}{\ln 10}$, proving the sum 
invariance 
\cite{Berger-2011}. 
Consequently, the $n$-digit Benford variable converges to Benford.


\end{document}